\documentclass[]{gSSR2e}
\newcommand{\R}{{\mathbb R}}
\newcommand{\Z}{\mathbb{Z}}
\newcommand{\D}{\mathcal{D}}
\newcommand{\E}{\mathbb E}
\newtheorem{assump}{Assumption}

\author{Jamil Salhi, James MacLaurin and Salwa Toumi}
\begin{document}
\title{  On Uniform Propagation of Chaos and application }
\maketitle{Abstract}
\begin{abstract}
In this paper we obtain time uniform propagation estimates for systems of interacting diffusion processes. 
Using a well defined metric function $h$ , our result guarantees a time-uniform estimates for the convergence of a class of interacting stochastic differential equations towards their mean field equation, and this for a general model, satisfying various conditions ensuring that the decay associated to the internal dynamics term dominates
the interaction and noise terms. Our result should have diverse applications, particularly in neuroscience, 
and allows for models more elaborate than the one of Wilson and Cowan, not requiring the internal dynamics to be of 
linear decay. An example is given at the end of this work as an illustration of the interest of this result.
\end{abstract}
\footnote{Contact: jamil.salhi@enit.rnu.tn, j.maclaurin@sydney.edu.au, salwa.toumi@lamsin.rnu.tn}
 \bigskip

\begin{keywords} Stochastic differential equation, Mean Fields,  McKean-Vlasov equations, Interacting Diffusion,
 Uniform Propagation of Chaos, Neural Network.
\end{keywords}

\begin{classcode}  60K35, 60J60, 60J65, 92B20	\end{classcode}\bigskip

\section{Introduction}
In this paper we obtain time-uniform estimates for the convergence of a class of interacting diffusion stochastic
differential equations towards the associated mean field equation. The propagation of chaos resulting from this convergence 
when the number of particles $N$ tends to infinity is uniform in time which means that not only the particles are independent  of each other, but also this independence is reached uniformly in time. 

The $N$-particle interacting diffusion model is of the following form
\begin{equation}
X^j_t = x_{ini} + \int_0^t b_0(X^j_s) + \frac{1}{N}\sum_{k=1}^N b_1(X^j_s,X^k_s) ds + \int_0^t b_2(X^j_s)dW^j_s.\label{eqn base system}
\end{equation}
Here $(W^j)_{j\in\Z^+}$ are independent Wiener Processes, $x_{ini}$ is a constant and $b_0,b_2: \R \to \R$,
$b_1: \R\times\R \to \R$, 
 are measurable functions. We will explain further below our reasons for studying this type of model.
 For a probability measure on $\R$, $\gamma$, write $\bar{b}_1(x,\gamma) = \int_{\R}b_1(x,y)d\gamma(y)$. 
The limiting processes $(\bar{X}^j_t)$ are defined to be
\begin{equation}
\bar{X}^j_t = x_{ini} + \int_0^t b_0(\bar{X}^j_s) + \bar{b}_1(\bar{X}^j_s,\bar{\mu}_s)ds + \int_0^t b_2(\bar{X}^j_s)dW_s\label{eqn limit system},
\end{equation}
where $\bar{\mu}_s$ is the law of $\bar{X}^j_s$.
 The classical propagation of chaos result states that, under suitable conditions 
on $b_0,b_1$ and $b_2$, the probability law of $X^j_t$ over some fixed time interval $[0,T]$
 (this being a probability law on $C([0,T];\R)$), converges weakly to the probability law of
 $\bar{X}^j_t$. 
Refer to \cite{sznitman:91,baladron-fasoli-etal:12b,bossy-faugeras-etal:15}
for more details.
We briefly consider the following toy model to motivate our problem.  Consider for the moment the system
\begin{equation}
Y^j_t = y_{ini} + \frac{1}{N}\int_0^t \sum_{k=1}^N b_1(Y^j_s,Y^k_s)ds + W^j_t,
\end{equation}
where $b_1$ has Lipschitz constant $\tilde{b}^{Lip}_N$. Define
\begin{equation}
\bar{Y}^j_t = Y_{ini} +\frac{1}{N} \int_0^t \int_{\R}b_1(\bar{Y}^j_s,y)d\tilde{\mu}_s(y)ds + W^j_t,
\end{equation}
where $\tilde{\mu}_s$ is the law of $\bar{Y}^j_s$. Assume that both of the above equations have strong solutions.
 Using Gronwall's Inequality and the Cauchy-Schwartz inequality, \cite{sznitman:91} obtained a bound of the form 
\[
\E\left[ \sup_{t\in [0,T]}|Y^j_t - \bar{Y}^j_t|\right] \leq\exp\left( 2 T \tilde{b}^{Lip}_N\right)\sup_{s\in [0,T]}\E\left[b_1(\bar{Y}^j_s,\bar{Y}^k_s)^2\right]^{\frac{1}{2}}.
\]
It is clear from the above that $\bar{Y}^j_t$ is a good approximation to $Y^j_t$ when $NT\tilde{b}^{Lip}_N \ll 1$. 
It is also clear that as $T\to \infty$, this bound becomes very poor, particularly due to the exponentiation. In much modeling of interacting diffusions, such as neuroscience, it is difficult to assume that $T$ is small: indeed, often it is difficult to properly model the `start' of a system. It is therefore desirable to obtain convergence results which are uniform in time. This is the focus of this paper.

For $x,y\in \R$, let
\begin{equation}
h(x,y) = g(x)g(y)f(x-y),
\end{equation}
for some functions $f\geq 0$ and $g\geq 1$ described further below. We expect (but do not require) $f$ to be of the form $f(z) = f_{const} z^{2k}$ where $k$ is a positive integer. $f$ modulates 
the rate of convergence for when $X^j_t$ is `close' to $\bar{X}^j_t$. $g\geq 1$ is a weight function which modulates the
 behavior for when $|X^j_t|$ or $|\bar{X}^j_t|$ asymptote to $\infty$. If $h$ is a metric, then this result guarantees that
 the Wasserstein Distance (with respect to $h$) between the laws of $X^j_t$ and $\bar{X}^j_t$ converges to zero 
as $N\to\infty$, with a rate which is uniform in $t$.
As a consequence of Theorem \ref{theorem major result}, and since $h$ is a metric, the result guarantees that the joint
 law of any finite set of neurons(or particles in the  general case)
converges to a tensor product of iid processes, each with law given by the SDE in (\ref{eqn limit system}).
 It is easily verified
as explained in Corollary \ref{corollary1}.
 To the best of our knowledge, the first work on 
uniform propagation of chaos was 
\cite{moral-miclo:00}
 when approximating Feynman Kac Formula for non linear filtering. Other authors applied Log-Sobolev inequalities and 
concentration inequalities  \cite{malrieu:01,carrillo-mccann-etal:03,veretennikov:06,cattiaux-guillin-etal:08,moral-rio:11,
bolley-gentil-etal:13,moral-tugaut:14}.
 Most of the previously cited works assume that the interaction term is of the 
form $b_1(x,y) = \nabla F(x-y)$ and the local term is of the form $\hat{b}_0 = \nabla V$ for some $F,V$ satisfying 
certain convexity properties. This work is essentially a generalization of \cite{veretennikov:06}.

We are motivated in particular by the application of these models to neuroscience 
(see for instance \cite{wilson-cowan:72,hansel-sompolinsky:96,gerstner-kistler:02b,deco-jirsa-etal:08,destexhe-sejnowski:09,coombes:10,touboul-ermentrout:11,bressloff:12,touboul:14})
although we expect in fact that these results are applicable in other domains such as agent-based modeling in finance, 
insect-swarms, granular models and various other applications of statistical physics. We have been able to weaken some 
of the requirements in \cite{veretennikov:06} and other works, so that the results may be applied in arguably more biologically 
realistic contexts. We do not assume that the interaction term $b_1(x,y)$ is a function of $x-y$, as in many 
of the previously cited works. The uniform propagation of chaos result is essentially due to the stabilizing effect of
 the internal dynamics ($b_0$ term) outweighing the destabilizing effect of the inputs from other neurons ($b_1$ term) and 
the noise ($b_2$ term). In \cite{veretennikov:06}, it was assumed that the gradient of $b_0$ is always negative, and is at least 
linear. However it is not clear (at least in the context of neuroscience) that the decay resulting from the internal
 dynamics term is always this strong for large values of $|X^j_t|$. Neuroscientific models are only experimentally validated 
over a finite parameter range, and therefore it is not certain how to model the dynamics for when the state variable
 $X^j_t$ is very large or small. Our more abstract setup does not require the decay to be linear (as in for example the
 Wilson-Cowan model) for large values of $|X^j_t|$: indeed the decay could be sub-linear or super-linear; all that is
 required is that in the asymptotic limit the decay from $b_0$ dominates the destabilizing effects of $b_1$ and $b_2$.
 Another improvement of our model over \cite{veretennikov:06} is that we consider multiplicative noise (i.e. $b_2 \neq 1$).
 This is more realistic because we expect the noise term $ \int_0^t b_2(X^j_s)dW^j_s$ to be of decreasing influence
 as $|X^j_t|$ gets large. This is because one would expect in general that the system is less responsive to the noise when its activity is greatly elevated, since the system should be stable. The point is that experimentalists should have some liberty in fitting our model to experimental 
data; all that is required is that in the asymptotic limit the decay from $b_0$ dominates the destabilizing effects of $b_1$
 and $b_2$.

We do not delve into the details of existence and uniqueness of solutions, and so throughout we assume that
\begin{assump}\label{assumption one}
There exist unique strong solutions to \eqref{eqn base system} and \eqref{eqn limit system}.
\end{assump}

Our major result is the following uniform convergence property.
\begin{theorem}\label{theorem major result}
If Assumption \ref{assumption one} and the assumptions in Section \ref{Sect assumptions} hold, then there exists a constant $K$ such that for all $t\geq 0$
\[
\E\left[ h(X^j_t,\bar{X}^j_t)\right] \leq K N^{-\frac{a}{q(a-1)}},
\]
for integers $a>1$ and $q \geq 1$.

\end{theorem}

It is easy to show existence and uniqueness if, for example, $b_0,b_1$ and $b_2$ are each globally Lipschitz. In the case
 of existence and uniqueness of \eqref{eqn base system},   \cite[Theorem 3.6]{mao:08} provides a useful general 
criterion. Refer to \cite{bossy-faugeras-etal:15} for a discussion of how to treat the existence and uniqueness of \eqref{eqn limit system}
 in a more general case. 

Our paper is structured as follows. In Section \ref{Sect assumptions} we outline the assumptions of our model, in Section
 \ref{Section Proof} we prove Theorem \ref{theorem major result} and in Section \ref{application}
 we outline an example of a system satisfying the assumptions of Section \ref{Sect assumptions}.

\section{Assumptions}\label{Sect assumptions}
The requirements outlined below might seem quite tedious. However in the next section we consider an
 application which allows us to simplify many of them. We split $\R$ into two domains $\D$ and $\D^c$. $\D\subset
 \R$ is a closed compact interval which we expect the system to be most of the time. Over $\D$, we require that the 
natural convexity of $b_0$ dominates that of $b_1$ and $b_2$. In $\D^c$ we require bounds for when the absolute
 values of the variables are asymptotically large.

Assume that $f \geq 0$, that $f(z) = f(-z)$, $g\geq 1$ and $f(z) = 0$ if and only if $z=0$. Suppose that for $z\in\D$, $g(z) = 1$ and clearly
 $g'(z) = 0$. Write $\hat{b}_0 = - b_0$.
Assume that for all $x,y\in \R$, 
\begin{equation}
f'(x-y)\left(\hat{b}_0(x) - \hat{b}_0(y)\right) -\frac{1}{2}f''(x-y)(b_2(x) - b_2(y))^2\geq 0.\label{eq I1 positive}
\end{equation}
Assume that for all $x,y\in \D$, there exists a constant $c_0 > 0$ such that
\begin{equation}
f'(x-y)\left(\hat{b}_0(x) - \hat{b}_0(y)\right) -\frac{1}{2}f''(x-y)(b_2(x) - b_2(y))^2\geq c_0 f(x-y).\label{eq I1 uniform bound}
\end{equation}
Assume that for all $z\in \R$, there is some $a>1$ such that 
\begin{equation}
f'(z)^a \leq f(z).\label{eq f differential bound}
\end{equation}
Assume that there exists a constant $a_0\in \R $ such that for all $x\notin \D$,
\begin{equation}
\left| \frac{g'(x)}{g(x)}b_2(x)\right| \leq a_0.\label{eq a0 bound}
\end{equation}
Assume that for $x\notin \D$, for all probability measures $\gamma$ and all $y\in\R$, 
\begin{equation}
\frac{g'(x)}{g(x)}\left(\hat{b}_0(x)-\bar{b}_1(x,\gamma)-\frac{f'(x-y)}{f(x-y)}b_2(x)(b_2(x)-b_2(y))-\frac{1}{2}a_0 b_2(x)\right) \geq c_0.\label{large x g bound}
\end{equation}
\begin{equation}
\frac{g''(x)}{g(x)}b_2(x)^2\leq c_2.\label{eq c2 bound}
\end{equation}
Assume that there exist constants $\breve{c}_1,\grave{c}_1\in \R$ such that for all $x,y_1,y_2 \in \R$,
\begin{align}
g(x)^{\frac{2(a-1)}{a}}(b_1(x,y_1)-b_1(x,y_2)) \leq & \breve{c}_1 g(y_1)^{\frac{a-1}{a}} g(y_2)^{\frac{a-1}{a}}
f(y_1-y_2)^{\frac{a-1}{a}}.\label{eq c1 bound 1}\\
\left| b_1(y_1,x) - b_1(y_2,x)\right| \leq& \grave{c}_1f(y_1-y_2)^{\frac{a-1}{a}}.\label{eq c1 bound 2}
\end{align}
Assumption \eqref{eq c1 bound 1} might seems a little strange. If $g(x) \to \infty$ as $x\to \infty$, in the context of neuroscience it would mean that the relative influence of neuron $k$ on neuron $j$ decreases as $X^j_t \to \infty$. This seems biologically reasonable.
We assume that $c_0$ dominates the other terms, i.e.
\begin{equation}
c := c_0 - \breve{c}_1 -\grave{c}_1 - c_2 > 0.\label{eq c definition}
\end{equation}
For some positive integer $q>2$, we require that there exists a constant $C_2$ such that for all $s > 0$,
\begin{align}
\E\left[ \bar{b}_1(\bar{X}_s,\bar{\mu}_s)^q\right] \leq C_2.\label{eq b1 q bound}
\end{align}
Assume that there exists a constant $C_1 > 0$ such that for all $s>0$ and for all $N$,
\begin{equation}
\E\left[ g(\bar{X}_s)^{\frac{2(a-1)q}{aq-a-q}}\right],\E\left[ g(X_s)^{\frac{2(a-1)q}{aq-a-q}}\right] \leq C_1.\label{eq g bound}
\end{equation} 
\section{Proof of Theorem \ref{theorem major result}}\label{Section Proof}

We now outline the proof of Theorem \ref{theorem major result}.
\begin{proof}
We will prove that there exists a constant $C$ such that
\begin{equation}\label{eq final C }
\E\left[ h(X^j_t,\bar{X}^j_t)\right] \leq \int_0^t -c \E\left[ h(X^j_s,\bar{X}^j_s)\right] + CN^{-\frac{1}{q}}\E\left[ h(X^j_s,\bar{X}^j_s)\right]^{\frac{1}{a}}ds.
\end{equation}
The theorem will then follow from the application of Lemma \ref{lemma ut} to the above result.

We observe using Ito's Lemma that
\begin{equation}\label{Ito}
h(X^j_t,\bar{X}^j_t) = I_1 + I_2' + I_2'' + I_3 + I_4 + I_5 +  \int_0^t \left( \frac{\partial h}{\partial x}b_2(X^j_s) + \frac{\partial h}{\partial y}b_2(\bar{X}^j_s)\right)dW^j_s.
 \end{equation}
The $I_j$ are 
\begin{align}
I_1 = &\int_0^t -g(X^j_s)g(\bar{X}^j_s)f'(X^j_s - \bar{X}^j_s)\left(\hat{b}_0(X^j_s) - \hat{b}_0(\bar{X}^j_s)\right)\nonumber\\
&+ \frac{1}{2} g(X^j_s)g(\bar{X}^j_s)f''\left(X^j_s - \bar{X}^j_s\right)(b_2(X^j_s) - b_2(\bar{X}^j_s))^2 ds \label{eq I1}\\
I'_2 = &\int_0^t f(X^j_s - \bar{X}^j_s)g(X^j_s)g'(\bar{X}^j_s)(b_1(\bar{X}^j_s,\bar{\mu}_s) -\hat{b}_0(\bar{X}^j_s))\nonumber\\ &
-f'(X^j_s - \bar{X}^j_s)g(X^j_s)g'(\bar{X}^j_s)b_2(\bar{X}^j_s)(b_2(\bar{X}^j_s) - b_2(X_s^j))ds\\
I_2'' =&\int_0^t f(X^j_s - \bar{X}^j_s)g'(X^j_s)g(\bar{X}^j_s)\left(\bar{b}_1(X^j_s,\hat{\mu}_s) - \hat{b}_0(X^j_s)\right)\nonumber \\ 
&+f'(X^j_s - \bar{X}^j_s)g'(X^j_s)g(\bar{X}^j_s)b_2(X^j_s)(b_2(X^j_s) - b_2(\bar{X}_s^j)) ds\label{eq I2} \\
I_3 =&\int_0^t g(X^j_s)g(\bar{X}^j_s)f'(X^j_s - \bar{X}^j_s)\left(\bar{b}_1(X^j_s,\hat{\mu}) - \bar{b}_1(\bar{X}^j_s,\bar{\mu}_s)\right)ds \label{eq I3}\\ 
I_4 = &\int_0^t f(X^j_s - \bar{X}^j_s)g'(X^j_s)g'(\bar{X}^j_s)b_2(X^j_s)b_2(\bar{X}^j_s)ds \label{eq I4}\\
I_5 =&\frac{1}{2}\int_0^t f(X^j_s - \bar{X}^j_s)\left( g''(X^j_s)g(\bar{X}^j_s) b_2(X^j_s)^2 + g''(\bar{X}^j_s)g(X^j_s) b_2(\bar{X}^j_s)^2\right) ds \label{eq I5}
\end{align}
We start by establishing that
\begin{equation}\label{I1 bound}
E\left[I_1+I'_2+I''_2 + I_4 \right] \leq -c_0 \int_0^t \E\left[ h(X^j_s,\bar{X}^j_s)\right]ds.
\end{equation}
We prove that the sum of the integrands of $I_1$,$I_2'$,$I_2''$ and $I_4$ is less than or equal to $-c_0 h(X^j_s,\bar{X}^j_s)$. Suppose firstly that $X^j_s,\bar{X}^j_s \in \D$. Then the integrands of $I_2',I_2''$ and $I_4$ are all zero. Furthermore, using \eqref{eq I1 uniform bound}, the integrand of $I_1$ satisfies the bound
\begin{multline*}
-g(X^j_s)g(\bar{X}^j_s)f'(X^j_s - \bar{X}^j_s)\left(\hat{b}_0(X^j_s) - \hat{b}_0(\bar{X}^j_s)\right)\\
+ \frac{1}{2} g(X^j_s)g(\bar{X}^j_s)f''\left(X^j_s - \bar{X}^j_s\right)(b_2(X^j_s) - b_2(\bar{X}^j_s))^2 
\leq -c_0 h(X^j_s,\bar{X}^j_s).
\end{multline*}
Now suppose that $\bar{X}^j_s\notin\D$. The integrand of $I_1$ is less than or equal to zero because of \eqref{eq I1 positive}. Through  \eqref{eq a0 bound},
\begin{multline}
g'(\bar{X}^j_s)\bigg[ \big(\hat{b}_0(\bar{X}^j_s) - \bar{b}_1(\bar{X}^j_s,\hat{\mu}_s)\big)f(\bar{X}^j_s - X^j_s) - f'(\bar{X}^j_s - X^j_s)b_2(\bar{X}^j_s)\big(b_2(\bar{X}^j_s) - b_2(X^j_s)\big) \\- \frac{a_0}{2}b_2(\bar{X}^j_s)f(\bar{X}^j_s - X^j_s)\bigg] \geq g(\bar{X}^j_s)c_0 f(\bar{X}^j_s - X^j_s).
\end{multline}
Since $f(-z) = f(z)$ and $f'(-z) = -f(z)$, upon multiplying the above identity by $-g(X^j_s)$,
\begin{multline}
g(X^j_s)g'(\bar{X}^j_s)\left[f(X^j_s - \bar{X}^j_s)(\bar{b}_1(\bar{X}^j_s,\hat{\mu}_s) -\hat{b}_0(\bar{X}^j_s)) +\right.\\ \left.f'(X^j_s - \bar{X}^j_s)b_2(\bar{X}^j_s)(b_2(X^j_s) - b_2(\bar{X}_s^j))\right] +\frac{1}{2} f(X^j_s - \bar{X}^j_s)g'(X^j_s)g'(\bar{X}^j_s)b_2(X^j_s)b_2(\bar{X}^j_s)\\ 
\leq -c_0 h(X^j_s,\bar{X}^j_s)-\frac{1}{2}a_0g(X^j_s)g'(\bar{X}^j_s)b_2(\bar{X}^j_s) f(X^j_s - \bar{X}^j_s)+\\ \frac{1}{2} f(X^j_s - \bar{X}^j_s)g'(X^j_s)g'(\bar{X}^j_s)b_2(X^j_s)b_2(\bar{X}^j_s)
\leq - c_0  h(X^j_s,\bar{X}^j_s),\label{int temp 3}
\end{multline}
since by \eqref{large x g bound},
\begin{equation*}
\frac{1}{2}a_0g(X^j_s)g'(\bar{X}^j_s)b_2(\bar{X}^j_s) f(X^j_s - \bar{X}^j_s)-\frac{1}{2} f(X^j_s - \bar{X}^j_s)g'(X^j_s)g'(\bar{X}^j_s)b_2(X^j_s)b_2(\bar{X}^j_s) \geq 0.
\end{equation*}

Notice that the left hand side of \eqref{int temp 3} is the sum of the integrand of $I_2'$ and half of the integrand of $I_4$. Similarly if $X^j_s\notin\D$, the integrand of $I_1$ is less than or equal to zero, and through \eqref{eq a0 bound} and \eqref{large x g bound},
\begin{multline}
g'(X^j_s)g(\bar{X}^j_s)\left[f(X^j_s - \bar{X}^j_s)(\bar{b}_1(X^j_s,\hat{\mu}_s) -\hat{b}_0(X^j_s)) +\right.\\ \left.f'(X^j_s - \bar{X}^j_s)b_2(X^j_s)(b_2(X^j_s) - b_2(\bar{X}_s^j))\right] +\frac{1}{2} f(X^j_s - \bar{X}^j_s)g'(X^j_s)g'(\bar{X}^j_s)b_2(X^j_s)b_2(\bar{X}^j_s)\\ 
\leq -c_0 h(X^j_s,\bar{X}^j_s)-\frac{1}{2}a_0g'(X^j_s)g(\bar{X}^j_s)b_2(X^j_s) f(X^j_s - \bar{X}^j_s)+\\ \frac{1}{2} f(X^j_s - \bar{X}^j_s)g'(X^j_s)g'(\bar{X}^j_s)b_2(X^j_s)b_2(\bar{X}^j_s)
\leq - c_0  h(X^j_s,\bar{X}^j_s).\label{int temp 4}
\end{multline}
The left hand side of the above is equal to the integrand of $I_2''$ and half of the integrand of $I_4$. Observe that if $\bar{X}^j_s \in \D$, then the left hand side of \eqref{int temp 3} is zero because $g'$ is zero in $\D$. Similarly if $X^j_s \in \D$, then the left hand side of \eqref{int temp 4} is zero because $g'$ is zero in $\D$. These considerations yield the bound \eqref{I1 bound}.

It follows from \eqref{eq c2 bound} that
\begin{equation*}
\E\left[ I_5\right] \leq c_2 \int_0^t \E\left[ h(X^j_s,\bar{X}^j_s)\right]ds.
\end{equation*}

We finish by bounding the $I_3$ term. Suppose that $g(X^j_s) \geq g(\bar{X}^j_s)$. Then using \eqref{eq f differential bound}, \eqref{eq c1 bound 1}-\eqref{eq c1 bound 2} and the triangular inequality
\begin{multline*}
\left| f'(X^j_s - \bar{X}^j_s)g(X^j_s)g(\bar{X}^j_s)\left(b_1(X^j_s,X^k_s) - b_1(\bar{X}^j_s,\bar{X}^k_s)\right)\right| \leq\\
\left| f'(X^j_s - \bar{X}^j_s)\right| g(X^j_s)g(\bar{X}^j_s)\left|b_1(X^j_s,\bar{X}^k_s)- b_1(\bar{X}^j_s,\bar{X}^k_s)\right|\\+ 
\left| f'(X^j_s - \bar{X}^j_s)\right| g(X^j_s)g(\bar{X}^j_s)\left| b_1(X^j_s,X^k_s) - b_1(X^j_s,\bar{X}^k_s)\right|\\
\leq \grave{c}_1 \left| f'(X^j_s - \bar{X}^j_s)\right| g(X^j_s)g(\bar{X}^j_s)f(X^j_s - \bar{X}^j_s)^{\frac{a-1}{a}}\\+\left| f'(X^j_s - \bar{X}^j_s)\right| g(X^j_s)^{\frac{1}{a}}g(\bar{X}^j_s)^{\frac{1}{a}}g(X^j_s)^{2-\frac{2}{a}}\left| b_1(X^j_s,X^k_s) - b_1(X^j_s,\bar{X}^k_s)\right|\\
\leq \grave{c}_1 g(X^j_s)g(\bar{X}^j_s)f(X^j_s - \bar{X}^j_s) +\\ g(X^j_s)^{\frac{1}{a}}g(\bar{X}^j_S)^{\frac{1}{a}}f(X^j_s-\bar{X}^j_s)^{\frac{1}{a}} g(X^j_s)^{2-\frac{2}{a}}\left| b_1(X^j_s,X^k_s) - b_1(X^j_s,\bar{X}^k_s)\right| \\
\leq \grave{c}_1 g(X^j_s)g(\bar{X}^j_s)f(X^j_s - \bar{X}^j_s) +\\ \breve{c}_1f(X^j_s - \bar{X}^j_s)^{\frac{1}{a}}g(X^j_s)^{\frac{1}{a}}g(\bar{X}^j_s)^{\frac{1}{a}}g(X^k_s)^{\frac{a-1}{a}}g(\bar{X}^k_s)^{\frac{a-1}{a}}f(X^k_s - \bar{X}^k_s)^{\frac{a-1}{a}}.
\end{multline*}
We obtain the same inequality when $g(\bar{X}^j_s) \geq g(X^j_s)$. That is,
\begin{multline*}
\left| f'(X^j_s - \bar{X}^j_s)g(X^j_s)g(\bar{X}^j_s)\left(b_1(X^j_s,X^k_s) - b_1(\bar{X}^j_s,\bar{X}^k_s)\right)\right| \leq\\
\left| f'(X^j_s - \bar{X}^j_s)\right| g(X^j_s)g(\bar{X}^j_s)\left|b_1(\bar{X}^j_s,X^k_s)- b_1(\bar{X}^j_s,\bar{X}^k_s)\right|\\+ 
\left| f'(X^j_s - \bar{X}^j_s)\right| g(X^j_s)g(\bar{X}^j_s)\left| b_1(X^j_s,X^k_s) - b_1(\bar{X}^j_s,X^k_s)\right|\\
\leq \grave{c}_1 g(X^j_s)g(\bar{X}^j_s)f(X^j_s - \bar{X}^j_s) +\\ \breve{c}_1f(X^j_s - \bar{X}^j_s)^{\frac{1}{a}}g(X^j_s)^{\frac{1}{a}}g(\bar{X}^j_s)^{\frac{1}{a}}g(X^k_s)^{\frac{a-1}{a}}g(\bar{X}^k_s)^{\frac{a-1}{a}}f(X^k_s - \bar{X}^k_s)^{\frac{a-1}{a}}.
\end{multline*}
Applying Holder's Inequality to the above,
\begin{multline*}
\E\left[f'(X^j_s - \bar{X}^j_s)g(X^j_s)g(\bar{X}^j_s)\left(b_1(X^j_s,X^k_s) - b_1(\bar{X}^j_s,\bar{X}^k_s)\right)\right] \leq\\
\grave{c}_1 \E\left[g(X^j_s)g(\bar{X}^j_s)f(X^j_s - \bar{X}^j_s)\right] +\\ \breve{c}_1 \E\left[g(X^j_s)g(\bar{X}^j_s)f(X^j_s-\bar{X}^j_s)\right]^{\frac{1}{a}}\E\left[g(X^k_s)g(\bar{X}^k_s)f(X^k_s - \bar{X}^k_s)\right]^{\frac{a-1}{a}} \\
=(\grave{c}_1 + \breve{c}_1)\E\left[g(X^j_s)g(\bar{X}^j_s)f(X^j_s - \bar{X}^j_s)\right] 
\end{multline*}
We use Holder's Inequality to see that
\begin{multline*}
\E\left[ f'(X^j_s - \bar{X}^j_s)g(X^j_s)g(\bar{X}^j_s)\left(\sum_{k=1}^N b_1(\bar{X}^j_s,\bar{X}^k_s) - \bar{b}_1(\bar{X}^j_s,\bar{\mu}_s)\right)\right] \leq \\
\E\left[ f'(X^j_s - \bar{X}^j_s)^a g(X^j_s)g(\bar{X}^j_s)\right]^{\frac{1}{a}} \E\left[ g(\bar{X}_s^j)^{\frac{a-1}{a}\times\frac{2aq}{aq-a-q}}\right]^{\frac{aq-a-q}{2aq}}\times\\ \E\left[ g(X_s^j)^{\frac{a-1}{a}\times\frac{2aq}{aq-a-q}}\right]^{\frac{aq-a-q}{2aq}}\times \E\left[\left(\sum_{k=1}^N b_1(\bar{X}^j_s,\bar{X}^k_s) - \bar{b}_1(\bar{X}^j_s,\bar{\mu}_s)\right)^q\right]^{\frac{1}{q}}.
\end{multline*}
where $q$ is the integer that appears in assumption \eqref{eq b1 q bound}.\\
By Assumption \eqref{eq g bound}, 
\[\E\left[ g(\bar{X}_s^j)^{\frac{2(a-1)q}{aq-a-q}}\right]^{\frac{aq-a-q}{aq}}\times \E\left[ g(X_s^j)^{\frac{2(a-1)q}{aq-a-q}}\right]^{\frac{aq-a-q}{aq}}\]
is uniformly bounded for all $s$. Furthermore through Assumption \eqref{eq b1 q bound} and Lemma
 \ref{lemma bound polynomial}, $\E\left[\left(\sum_{k=1}^N b_1(\bar{X}^j_s,\bar{X}^k_s) - \bar{b}_1(\bar{X}^j_s,\bar{\mu}_s)
\right)^q\right]^{\frac{1}{q}}$ is bounded by $\mathfrak{C} N^{\frac{q-1}{q}}$. Finally, using Assumption \eqref{eq f differential bound},
\[
\E\left[ f'(X^j_s - \bar{X}^j_s)^a g(X^j_s)g(\bar{X}^j_s)\right]^{\frac{1}{a}} \leq \E\left[h(X^j_s,\bar{X}^j_s)
\right]^{\frac{1}{a}}.
\]
We thus find that for some constant $C$,
\begin{equation*}
\E\left[ I_3\right] \leq C\int_0^t N^{-\frac{1}{q}}\E\left[h(X^j_s,\bar{X}^j_s)\right]^{\frac{1}{a}}ds.
\end{equation*}
In summary, noting the assumption \eqref{eq c definition}, we now have all the ingredients for \eqref{eq final C }.
\end{proof}
\begin{corollary}\label{corollary1}
Let $l \in  \mathbb{N}^{*}$ and fix $l$ neurons $(i_{1},...,i_{l})\in  \mathbb{N}^{*}$. Under the assumptions of Theorem 1, 
 the law of $(X^{i_{1}}_{t},...,X_{t}^{i_{l}} )$, converges toward $\mu_{t}^{\otimes l}$ for all $t\geq 0$.
 \end{corollary}
\begin{proof}
\[
\E\left[ \left|(X^{i_{1}}_{t},...,X^{i_{l}}_{t})-(\bar{X}^{i_{1}}_{t},...,\bar{X}^{i_{l}}_{t})\right|^2\right]
  \leq \sum^{l}_{k=1} \E\left[ \left|X^{i_{k}}_{t}-\bar{X}^{i_{k}}_{t}\right|^2\right]  
	\leq l K N^{-\frac{a}{q(a-1)}},\]
Hence  $\forall t \geq 0$ the law of $( X^{i_{1}}_t,...,X^{i_{l}}_t ) $ converges when $N$ tends to infinity to the law of $ (\bar{X}^{i_{1}}_t,...,\bar{X}^{i_{l}}_t)$  ,  whose law is equal to $\mu_{t}^{\otimes l}$ by definition.
\end{proof}
We present now the lemmas used in the proof of Theorem \ref{theorem major result}.
\begin{lemma}\label{lemma bound polynomial}
Suppose that $(e^j)_{j=1}^\infty$ are independent identically-distributed $\R$-valued random variables such that $\E\left[ (e^j)^q\right] < \infty$ and $\E[e^j] = 0$. Then there exists a constant $\mathfrak{C}$ such that for all $N$
\[
\E\left[ \left(\sum_{j=1}^N e^j\right)^q\right] < \mathfrak{C} N^{q-1}.
\]
\end{lemma}
\begin{proof}
Consider the binomial expansion of $ \left(\sum_{j=1}^N e^j\right)^q$. There are $N^q$ terms in total. The expectation of at least $N\times (N-1)\times (N-2)\times \ldots (N-q+1)$ of these must be zero, as the constituent factors are all independent. Let $(j_i)_{i=1}^q$, $1\leq j_i \leq N$, be an arbitrary set of indices. Then through Holder's Inequality,
\[
\E\left[ \prod_{p=1}^q e^{j_p}\right] \leq \E\left[ (e^1)^q\right].
\]
Thus
\begin{multline*}
\E\left[ \left(\sum_{j=1}^N e^j\right)^q\right] \leq\\ \E\left[ (e^1)^q\right]\times\left( N^q - N\times (N-1)\times (N-2)\times \ldots (N-q+1)\right) \\
\leq  \E\left[ (e^1)^q\right]\times \left( N^q - (N-q+1)^q\right)\\
=  \E\left[ (e^1)^q\right]N^{q-1}\left( N - N\left( 1 - \frac{q-1}{N}\right)^{q-1}\right) 
\leq \E\left[ (e^1)^q\right]N^{q-1}(q-1)(q-2).
\end{multline*}
\end{proof}

The following lemma is an easy generalization of a result in \cite{veretennikov:06}.
\begin{lemma}\label{lemma ut}
Suppose that $u$ is continuous and satisfies, for some constants $\mathcal{C},c > 0$ and positive integer $a > 1$, for all $t< T$, 
\[
u(T) - u(t) \leq \int_t^T -cu(s)+\mathcal{C}u(s)^{\frac{1}{a}}ds. 
\]
Furthermore $u(0) = 0$. Then for all $t\geq 0$
\[
u(t) \leq \left(\frac{\mathcal{C}}{c}\right)^{\frac{a}{a-1}}.
\]
\end{lemma}
\begin{proof}
It may be seen that $u$ is differentiable, with the derivative satisfying
\[
\dot{u}(t) \leq - c u(t) + \mathcal{C}u(t)^{\frac{1}{a}}.
\]
Let $v(t) = u(t)\exp(ct)$. Then
\begin{equation*}
\dot{v}(t) \leq \mathcal{C}v(t)^{\frac{1}{a}}v\exp\left( \frac{(a-1)ct}{a}\right).
\end{equation*}
If $v(t)=0$ then there is nothing to show. Thus we may assume that for all $t>0$, $v(t) > 0$. Hence
\[
\dot{v}(t) v(t)^{-\frac{1}{a}} \leq \mathcal{C}\exp\left( \frac{ct(a-1)}{a}\right).
\]
Upon integration,
\[
\frac{a}{a-1} v(t)^{\frac{a-1}{a}} \leq \frac{a}{a-1}\frac{\mathcal{C}}{c}\left( \exp\left(\frac{ct(a-1)}{a}\right) - 1 \right)\leq\frac{a}{a-1}\frac{\mathcal{C}}{c}\exp\left(\frac{ct(a-1)}{a}\right).
\]
Thus 
\[
v(t) \leq \left(\frac{\mathcal{C}}{c}\right)^{\frac{a}{a-1}}\exp(ct).
\]
\end{proof}
\section{Application}\label{application}
In this section we are going to provide 
 an example of a system satisfying the requirements of Section \ref{Sect assumptions}, so that the result of Theorem
 \ref{theorem major result} will apply.
We start by defining the following functions.\\
For all $x\in \R$, let $f(x):=\frac{1}{4}x^2.$ Let $\D=[-A,A]$ for some $A\gg 0$.We take $a=2$ and $q=3$. Define the sigmoid function
$S(x):=\frac{1}{1+exp(-x)}$, it is clear that $S$ is of class $C^\infty$ , $0<S(x)<1$ and its derivative is bounded and positive. Using this, we define
$$g(x)=\left\{\begin{array}{lll}
1 \qquad\mbox{if}\qquad x\in\D  \\
S(-A-x)+\frac{1}{2} \qquad\mbox{if}\qquad x<-A     \\
S(-A+x)+\frac{1}{2} \qquad\mbox{if}\qquad x>A 
\end{array}\right. 
$$

The function $g$ is continuous on $\R$, $1\leq g(x) \leq \frac{3}{2}$, its derivative $g'$ is bounded, negative for $x<-A$ and positive 
for $x>A$.

We consider a population of $N$ neurons, with evolution equation 
\begin{equation}
dV^j_t = (-\frac{1}{\tau}V^j_t + \frac{1}{N}\sum_{k=1}^N J(V_t^j,V_t^k)S(V_t^k)+I(t))dt +
 \sigma dW^j_t, \label{eqn example1}
\end{equation}
where $V_t^j$ is the membrane potential of neuron $j$, $I(t)$ is the deterministic input current. $J(V_t^j,V_t^k)$ denotes
 the synaptic weight from neuron $k$ to neuron $j$. The function $J: \R\times\R\rightarrow\R$  is assumed to be of class  $C^1$
in both variables, such that both it and its derivative are bounded. \\
The above assumptions are sufficient for the requirements of Section \ref{Sect assumptions} to be satisfied. In particular,
 using the Mean Value Theorem, one can easily verify the bounds \ref{eq c1 bound 1} and \ref{eq c1 bound 2}. 
Morever, one can refer to \cite{bolley-gentil-etal:13} and verify that assumption \ref{assumption one} is satisfied. It then follows, using Theorem \ref{theorem major result}, that for all $t>0$
\[
\E\left[ (V^j_t-\bar{V}^j_t)^2\right] \leq 4K N^{-\frac{2}{3}},
\]
In other words, the law of an individual neuron converges to its limit as $N\to \infty$ at the
 time-uniform rate given above.


\newpage

\begin{thebibliography}{12}
\bibitem[1]{baladron-fasoli-etal:12b}
J. Baladron, D. Fasoli, O. Faugeras, and J. Touboul, {\em  Mean-field description 
and propagation of chaos in networks of Hodgkin-Huxley and Fitzhugh-Nagumo neurons} , The journal 
Of Mathematical Neuroscience, 2 (2012).
\bibitem[2]{bolley-gentil-etal:13}
F. Bolley, I. Gentil, and A. Guillin, {\em Uniform convergence to equilibrium for granular 
media}, Archive for Rational Mechanics and Analysis, 208 (2013), pp. 429-445.
\bibitem[3]{bossy-faugeras-etal:15}
M. Bossy, O. Faugeras, and D. Talay, {\em Clarification and complement to ``mean-field description and propagation of
 chaos in networks of hodgkin-huxley and fitzhugh-nagumo neurons} , tech. report, HAL INRIA, 2015.
\bibitem[4]{bressloff:12}
P.C. Bressloff, {\em Spatiotemporal dynamics of continuum neural fields}, Journal of Physics 
A: Mathematical and Theoretical, 45 (2012).
\bibitem[5]{carrillo-mccann-etal:03}
J. A. Carillo, R. J. Mccann, and C. Villani,{\em  Kinetic equilibration rates for granular media and 
related equations: entropy dissipation and mass transportation estimates}, Revista Matematica Iberoamericana, 19 (2003), pp. 971-1018.
\bibitem[6]{cattiaux-guillin-etal:08}
P. Cattiaux, A. Guillin, and F. Malrieu, {\em Probabilistic approach for granular media equa-
tions in the non-uniformly convex case}, Probability Theory and Related Fields, (2008).
\bibitem[7]{coombes:10}
S. Coombes, {\em Large-scale neural dynamics: Simple and complex}, Neurolmage, 52 (2010),pp. 731-739.
\bibitem[8]{deco-jirsa-etal:08}
G. DECO,V. K. Jirsa, P. A. Robinson, M. Breakspear, and K. 
Friston, {\em The dynamic brain: From spiking neurons to neuralmasses and cortical fields},
PloS Comput. Biol., 4 (2008).
\bibitem[9]{destexhe-sejnowski:09}
A. Destexhe and T. J. Sejnowski,{\em  The Wilson-Cowan model}, 36 years later, 
Biological Cybernetics, 101 (2009), pp. 1-2.
\bibitem[10]{gerstner-kistler:02b}
W. Gerstner and W. Kistler, {\em Spiking Neuron Models}, Cambridge University Press, 2002.
\bibitem[11]{hansel-sompolinsky:96}
D. Hansel and H. Sompolinsky, {\em Chaos and synchrony in a model of a hypercolumn in 
visual cortex}, Journal of Computational Neuroscience, 3 (1996), pp. 7-34.
\bibitem[12]{malrieu:01}
F. Malrieu, {\em Logarithmic sobolev inequalities for some nonlinear pde's}, Stochastic Processes and their 
Applications, 95 (2001), pp. 109-132.
\bibitem[13]{mao:08}
X. Mao, {\em Stochastic differential equations and applications}, Horwood, 2008, 2nd Edition.
\bibitem[14]{moral-miclo:00}
P. Del Moral and L. Miclo, {\em Branching and interacting particle systems approximations of feynman-kac formulae
 with applications
to non-linear filtering,in Séminaire de Probabilités XXXIV} , J. Azéma, M. Emery, M. Ledoux, and M. Yor, eds., vol. 1729, Springer-
Verlag Berlin, 2000.
\bibitem[15]{moral-rio:11}
P. Del Moral and E. Rio, {\em Concentration inequalities for mean field particle models} , Annals 
of Applied Probability, (2011).
\bibitem[16]{moral-tugaut:14}
P. Del Moral and J. Tugaut, {\em Uniform propagation of chaos for a class of inhomogeneous diffusions},
tech. 
Report, HAL INRIA, 2014.
\bibitem[17]{sznitman:91}
A. Sznitman, {\em Topics in propagation of chaos, in Ecole d'été de Probabilités de Saint-Flour
 XIX-1989, Donald Burkholder,Etienne Pardoux, and Alain-Sol Sznitman, eds}, vol. 1464 of Lecture Notes in 
Mathematics, Springer Berlin / Heidelberg, 1991, pp. 165-251. 10.1007/BFb0085169.
\bibitem[18]{touboul:14} 
J. Touboul, {\em the propagation of chaos in neural fields}, The Annals of Applied Probability, 24 (2014). 
visual cortex, Journal of Computational Neuroscience, 3 (1996), pp. 7-34.
\bibitem[19]{touboul-ermentrout:11}
J. Touboul and B. Ermentrout, {\em  Finite-size and correlation-induced effects in 
mean-field dynamics}, J Comput Neurosci, 31 (2011), pp.453-484.
\bibitem[20]{veretennikov:06}
A. YU VERETENNIKOV,{\em  On ergodic measures for mckean-vlasov stochastic equations}, Monte-Carlo and Quasi-Monte-Carlo 
Methods, (2006), pp. 471-486.
\bibitem[21]{wilson-cowan:72}
H.R. Wilson and J.D. Cowan,{\em  Excitatory and inhibitory interactions in localized polulations 
of model neurons}, Biophys. J., 12 (1972),pp. 1-24.


\end{thebibliography}
\end{document}